\documentclass[11pt,reqno]{amsart}
\usepackage{verbatim}

\newtheorem{thm}{Theorem}[section]
\newtheorem{lem}[thm]{Lemma}
\newtheorem{prop}[thm]{Proposition}
\newtheorem{cor}[thm]{Corollary}

\theoremstyle{definition}

\theoremstyle{remark}
\newtheorem*{rem}{Remark}

\newcommand{\eps}{\varepsilon}

\numberwithin{equation}{section}
\usepackage{enumerate}
\begin{document}

\title[Geometry of Banach spaces with an octahedral norm]{Geometry of Banach spaces with \\an octahedral norm}

\author{Rainis Haller and Johann Langemets}
\address{Institute of Mathematics, University of Tartu, J.~Liivi 2, 50409 Tartu, Estonia}
\curraddr{}
\email{rainis.haller@ut.ee}
\email{johann.langemets@ut.ee}
\thanks{This research was supported by Estonian Targeted Financing Project SF0180039s08 and by Estonian Science Foundation Grant 8976.}

\subjclass[2010]{Primary 46B20, 46B22}

\keywords{Octahedral norm, thickness}

\date{}

\dedicatory{}

\begin{abstract}
We discuss the geometry of Banach spaces whose norm is octahedral or, more generally, locally or weakly octahedral. Our main results characterize these spaces in terms of covering of the unit ball.
\end{abstract}

\maketitle
\section{Introduction}\label{sec: 1. Introduction}
Let $X$ be a Banach space over the field of real numbers $\mathbb R$. A norm on $X$ is called \emph{octahedral} (e.g. \cite{G, DGZ, BGLPRZoca, HLP}) if, for every finite-dimensional subspace $E$ of $X$ and $\eps>0$ there exists an element $y$ of $X$ with $\|y\|=1$ and
\begin{equation*}\label{eq: oct_def}
\|x-y\|\geq(1-\varepsilon)(\|x\|+\|y\|)\quad\text{for all}\ x\in E.
\end{equation*}
Note that only infinite-dimensional Banach spaces may possess an octahedral norm.

In addition to the octahedral norms, two wider classes called locally octahedral norms and weakly octahedral norms (see definitions below in Section~\ref{sec: 3}) were introduced and studied in a very recent preprint \cite{HLP} (cf. \cite{BGLPRZoca}), whereas these three octahedrality conditions came up in relation to the diameter two properties \cite{ALN}.

In this paper, we shall study the geometry of those Banach spaces whose norm is octahedral or, more generally, locally or weakly octahedral. Our main results characterize these spaces in terms of covering of the unit ball.

The starting point of our investigation is the following observation made by G.~Godefroy in \cite{G}:
\medskip

\emph{The norm on $X$ is octahedral if and only if $B_X$, the closed unit ball of $X$, has the property that in every finite covering of $B_X$ with closed balls at least one member of the covering itself contains $B_X$.}

\medskip\noindent We shall recall the result in Proposition~\ref{prop: oct balls} below and provide its proof for completeness.
In the special case when $X$ is separable, the preceding equivalence is contained in \cite[Lemma~9.1]{MR983386}.

We consider only infinite-dimensional real Banach spaces. Let us fix some notation. By $B(x,r)$ we denote the closed ball in $X$ with center at $x\in X$ and radius $r>0$. The closed unit ball of a Banach space $X$ is denoted by $B_X$ and its unit sphere by $S_X$. The dual space of $X$ is denoted by $X^\ast$.

\section{Octahedral spaces}\label{sec: 2}
We recall a simple criterion for the octahedral norm.

\begin{lem}[{\cite[Proposition 2.4]{HLP}}]\label{lem: oct}
Let $X$ be a Banach space. The following assertions are equivalent:
\begin{enumerate}[\upshape (i)]
\item the norm on $X$ is octahedral;
\item for every $n\in\mathbb N$, $x_1,\dotsc,x_n\in S_X$, and $\varepsilon>0$ there is a $y\in S_X$ such that
\[
\|x_i-y\|\geq 2-\varepsilon\quad\text{for all $i\in\{1,\dotsc,n\}$}.
\]
\end{enumerate}
\end{lem}

The following characterization of octahedral norms is given by P.~L. Papini in \cite{MR1735844}. We shall give its proof for completeness. Our proof is not essentially different from the one in \cite{MR1735844}, but the benefit of Lemma~\ref{lem: oct} is clear.

\begin{prop}[{\cite[Theorem 2.1]{MR1735844}}]
Let $X$ be a Banach space. The following assertions are equivalent:
\begin{enumerate}[\upshape (i)]
\item the norm on $X$ is octahedral;
\item $\mu_2(X)=2$, where
\[
\mu_2(X)=\inf_{\substack{x_1,\dotsc,x_n\in S_X\\ n\in\mathbb N}}\sup_{x\in S_X} \frac 1n\sum_{i=1}^n\|x_i-x\|.
\]
\end{enumerate}
\end{prop}

\begin{proof}
(i) $\Rightarrow$ (ii). Assume that the norm on $X$ is octahedral. By Lemma~\ref{lem: oct} (ii), we have $\sup_{x\in S_X}\sum_{i=1}^n\|x_i-x\|=2n$ for every $x_1,\dotsc,x_n$ in $S_X$. Thus $\mu_2(X)=2$.

(ii) $\Rightarrow$ (i). Assume that $\mu_2(X)=2$. Let $x_1,\dotsc,x_n\in S_X$ and let $\varepsilon>0$.
Since $\mu_2(X)=2$, we have
\[
\sup_{x\in S_X}\frac 1n\sum_{i=1}^n\|x_i-x\|=2.
\]
It follows that there is an element $x$ in $S_X$ with \[\frac 1n\sum_{i=1}^n\|x_i-x\|\geq 2-\frac \varepsilon n.\] This yields
\[
\|x_i-x\|\geq 2-\eps\quad\text{for all $i\in\{1,\dotsc,n\}$}.
\]
By Lemma~\ref{lem: oct} (ii), the norm on $X$ is octahedral.
\end{proof}

For a Banach space $X$, R.~Whitley \cite{MR0228997} introduced the \emph{thickness} parameter $T(X)$, defined by
\[
T(X)=\inf\{\eps>0\colon \text{there exists a finite $\eps$-net for $S_X$ in $S_X$}\}.
\]
For finite-dimensional Banach spaces $X$, one has $T(X)=0$. If $X$ is an infinite-dimensional Banach space then Whitley showed that $1\leq T(X)\leq 2$.

\begin{rem}
In an infinite-dimensional Banach space $X$ a finite covering of $S_X$ with closed balls covers the entire unit ball $B_X$ (for the proof see e.g. \cite{MR1541729}, \cite[Proposition~3ter]{MR2583590}, \cite[Proposition~2]{MR2984117}). Therefore, for an infinite-dimensional $X$,
\[
T(X)=\inf\{\eps>0\colon \exists\{x_1,\dotsc,x_n\}\subset S_X \text{ s.t. } B_X\subset\bigcup_{i=1}^n B(x_i,\eps)\}.
\]
\end{rem}

We shall soon see that the octahedrality of the norm on $X$ is also characterized by the condition $T(X)=2$. In fact, this is a direct consequence of the following observation.

\begin{prop}[cf. {\cite[p.~12]{G}}]\label{prop: oct balls}
Let $X$ be a Banach space. The following assertions are equivalent:
\begin{enumerate}[\upshape (i)]
\item the norm on $X$ is octahedral;
\item if $x_1,\dotsc,x_n\in X$ and $r_1,\dotsc,r_n>0$ are such that
\[
S_X\subset \bigcup_{i=1}^nB(x_i,r_i)
\]
then $S_X\subset B(x_i,r_i)$ for some $i$ in $\{1,\dotsc,n\}$;
\item if $x_1,\dotsc,x_n\in S_X$ and $r_1,\dotsc,r_n>0$ are such that
\[
S_X\subset \bigcup_{i=1}^nB(x_i,r_i)
\]

then $S_X\subset B(x_i,r_i)$ for some $i$ in $\{1,\dotsc,n\}$.
\end{enumerate}
\end{prop}

\begin{proof}
(i) $\Rightarrow$ (ii). Assume that the norm on $X$ is octahedral, and consider a finite number of closed balls $ B(x_1,r_1),\dotsc, B(x_n,r_n)$ in $X$  with
\[
S_X\subset \bigcup_{i=1}^nB(x_i,r_i).
\]
Since the norm on $X$ is octahedral, for every $\eps>0$ there are $i$ in $\{1,\dotsc,n\}$ and $y$ in $ B(x_i,r_i)$ with
\[
\|x_i-y\|\geq(1-\varepsilon)(\|x_i\|+1),
\]
which yields $r_i\geq(1-\varepsilon)(\|x_i\|+1)$.
Consequently, $r_i\geq\|x_i\|+1$ for at least one $i$ in $\{1,\dotsc,n\}$.

(ii) $\Rightarrow$ (iii) is trivial.

(iii) $\Rightarrow$ (i). Assume that (iii) holds. Let $n\in\mathbb N$, $x_1,\dotsc,x_n\in S_X$, and $\varepsilon>0$. By Lemma~\ref{lem: oct} it suffices to find a $y\in S_X$ such that
\[
\|x_i-y\|\geq 2-\varepsilon\quad\text{for all $i\in\{1,\dotsc,n\}$}.
\]
Suppose that, on the contrary, for every $y\in S_X$ there is an $x_i$ such that $\|x_i-y\|<2-\varepsilon$.
Then \[S_X\subset \bigcup_{i=1}^n B(x_i,2-\varepsilon).\] By our assumption, $S_X\subset  B(x_i,2-\varepsilon)$ for some $i$ in $\{1,\dotsc,n\}$, which is a contradiction.

\end{proof}

Condition (iii) of Proposition~\ref{prop: oct balls} is clearly equivalent to $T(X)=2$.

\begin{cor}\label{cor: X oct iff T(X)=2}
Let $X$ be a Banach space. The norm on $X$ is octahedral if and only if $T(X)=2$.
\end{cor}

\begin{rem}
In general $T(X)\geq T(X^{**})$. Since $C[0,1]$ is octahedral and $C[0,1]^{**}$ fails to be octahedral (see \cite[p. 12]{G}), we have $T(X)>T(X^{**})$ for $X=C[0,1]$ (cf. last paragraph in \cite{CPS}).
\end{rem}

Banach spaces that admit octahedral norms are exactly the spaces containing isomorphic copies of $\ell_1$.

\begin{thm}[{\cite[Theorem~2.5, p.~106]{DGZ}}]\label{thm: DGZ}
Let $X$ be a Banach space. The following assertions are equivalent:
\begin{enumerate}[\upshape (a)]
\item $X$ contains a subspace isomorphic to $\ell_1$;
\item there is an equivalent octahedral norm on $X$.
\end{enumerate}
\end{thm}

Combining Theorem~\ref{thm: DGZ} and Corollary~\ref{cor: X oct iff T(X)=2} leads to the following result.

\begin{thm}[cf. {\cite[Theorem~1.2]{MR2844454}}]\label{thm: main}
A Banach space $X$ can be equivalently renormed to have thickness $T(X)=2$ if and only if $X$ contains an isomorphic copy of $\ell_1$.
\end{thm}

Theorem~\ref{thm: main} strengthens Theorem~1.2 of \cite{MR2844454}, which asserts that a separable Banach space $X$ can be equivalently renormed to have thickness $T(X)=2$ if and only if $X$ contains an isomorphic copy of $\ell_1$. Moreover, it justifies Corollary~6 in \cite{MR2984117}. It also improves Proposition~3.1 in \cite{MR1775391}, which asserts that a Banach space $X$ with $\mu_2(X)=T(X)=2$ contains an isomorphic copy of $\ell_1$.

In a recent preprint \cite{CPS}, the behavior of Whitley's thickness with respect to $\ell_p$-products is studied. If $X$ and $Y$ are Banach spaces and $1<p<\infty$, then $T(X\oplus_p Y)\leq \max\{T(X),T(Y)\}$ (see \cite[Theorem~2]{CPS}). We have the following estimation for $X\oplus_p Y$.

\begin{prop}\label{prop: T of p-sum}
If $X$ and $Y$ are Banach spaces and $1<p<\infty$, then \[T(X\oplus_p Y)\leq \sqrt[p]{\frac{(\sqrt[p]{2}+1)^p+1}{2}}.\]
Thus \[T(X\oplus_2 Y)\leq\sqrt{2+\sqrt 2}.\]
\end{prop}

\begin{rem}
This estimation is sharp since $T(\ell_1\oplus_2\ell_1)=\sqrt{2+\sqrt 2}$ (\cite[Lemma~2]{CPS}). On the other hand,
$T(\ell_p\oplus_p Y)=\sqrt[p]{2}$ (\cite[Proposition~1]{CPS}) which is strictly less than our estimation shows.
\end{rem}

\begin{proof}
Denote by $Z=X\oplus_p Y$ and $r=\sqrt[p]{\dfrac{(\sqrt[p]{2}+1)^p+1}{2}}$. Fix arbitrarily $x_1\in S_X$ and $y_1\in S_Y$. We will show that \[S_Z\subset B((x_1,0),r)\cup B((0,y_1),r).\]

Let $(x,y)\in S_Z$ and set
\[m=\min\{\|(x_1,0)-(x,y)\|_p,\|(0,y_1)-(x,y)\|_p\}.\]
Then
\begin{align*}
m^p&=\min\{\|x_1-x\|^p+\|y\|^p,\|x\|^p+\|y_1-y\|^p\}\\
&\leq\dfrac{\|x_1-x\|^p+\|y\|^p+\|x\|^p+\|y_1-y\|^p}{2}\\
&=\dfrac{\|(x_1-x,y_1-y)\|^p+1}{2}.
\end{align*}
Since \[\|(x_1-x,y_1-y)\|\leq\|(x_1,y_1)\|+\|(x,y)\|=\sqrt[p]{2}+1,\]
we get
\[
m^p\leq \dfrac{(\sqrt[p]{2}+1)^p+1}{2}.
\]
\end{proof}

\begin{cor}
If $X$ and $Y$ are Banach spaces and $1<p<\infty$, then $X\oplus_p Y$ is never octahedral.
\end{cor}

\begin{proof}
By Proposition~\ref{prop: T of p-sum} we have
\[
T(X\oplus_p Y)^p\leq \dfrac{(\sqrt[p]{2}+1)^p+1}{2}<2^p.
\]
The last inequality is easily obtained from the Minkowski's inequality by considering $(\sqrt[p]{2},0), (1,1)\in\mathbb R^2$.
\end{proof}

\section{Locally and weakly octahedral spaces}\label{sec: 3}
According to the terminology in \cite{HLP}, the norm on a Banach space $X$ is
\begin{itemize}
\item
\emph{locally octahedral} if, for every $x\in X$ and $\eps>0$, there is a $y\in S_X$ such that
\[
\|sx-y\|\geq(1-\eps)\bigl(|s|\|x\|+\|y\|\bigr)\quad\text{for all $s\in\mathbb R$;}
\]
\item
\emph{weakly octahedral} if, for every finite-dimensional subspace $E$ of~$X$, $x^\ast\in B_{X^\ast}$, and $\eps>0$,
there is a $y\in S_X$ such that
\[
\|x-y\|\geq(1-\eps)\bigl(|x^\ast(x)|+\|y\|\bigr)\quad\text{for all $x\in E$.}
\]
\end{itemize}

Our aim in this section is to find analogous results to Proposition~\ref{prop: oct balls} for locally and weakly octahedral spaces.

We recall a simple criterion for the locally octahedral norm.

\begin{lem}[{\cite[Proposition 2.1]{HLP}}]\label{lem: loh}
Let $X$ be a Banach space. The following assertions are equivalent:
\begin{enumerate}[\upshape (i)]
\item the norm on $X$ is locally octahedral;
\item for every $x\in S_X$ and $\eps>0$, there is a $y\in S_X$ such that
\[
\|x\pm y\|\geq 2-\eps.
\]
\end{enumerate}
\end{lem}

For locally octahedral spaces we have the following geometric characterization.

\begin{prop}\label{prop: loh balls}
Let $X$ be a Banach space. The following assertions are equivalent:
\begin{enumerate}[\upshape (i)]
\item the norm on $X$ is locally octahedral;
\item if $S_X\subset B(x,r)\cup B(-x,r)$ for some $x\in X$ and $r>0$ then $S_X\subset  B(x,r)$;
\item if $S_X\subset  B(x,r)\cup B(-x,r)$ for some $x\in S_X$ and $r>0$ then $S_X\subset  B(x,r)$.
\end{enumerate}
\end{prop}
\begin{proof}
(i) $\Rightarrow$ (ii). Assume that the norm on $X$ is locally octahedral, and $S_X\subset  B(x,r)\cup B(-x,r)$ for some $x\in X$ and $r>0$. We have to show that $S_X\subset  B(x,r)$. Suppose that, contrary to the claim, $r<\|x\|+1$. Let $\varepsilon>0$ be such that $r<(1-\varepsilon)(\|x\|+1)$. Since the norm on $X$ is locally octahedral, there is a $y\in S_X$ such that
\[
(1-\varepsilon)(\|x\|+1)\leq \|x\pm y\|.
\]
Thus $\|x\pm y\|>r$, which is a contradiction.

(ii) $\Rightarrow$ (iii) is trivial.

(iii) $\Rightarrow$ (i). Assume that (iii) holds. Let $x\in S_X$ and $\varepsilon>0$.
By Lemma~\ref{lem: loh} it suffices to find a $y\in S_X$ such that
\[
\|x\pm y\|\geq 2-\varepsilon.
\]

Suppose that for every $y\in S_X$ we have $\|x+y\|<2-\varepsilon$ or $\|x-y\|<2-\varepsilon$. Thus $S_X\subset B(x, 2-\varepsilon)\cup B(-x,2-\varepsilon)$. By the assumption, $S_X\subset  B(x,2-\varepsilon)$, which is a contradiction.
\end{proof}

Condition (iii) of Proposition~\ref{prop: loh balls} is clearly equivalent to $g'(X)=2$, where the
the constant $g'(X)$ is defined by
\[
g'(X)=\inf\{\eps>0\colon S_X\subset B(x,\eps)\cup B(-x,\eps)\text{ for some $x$ in $S_X$}\}.
\]
The interested reader can find more about this constant $g'(x)$ in \cite{MR2583590}, where Papini has compared it with Whitley's thickness constant.

\begin{cor}
Let $X$ be a Banach space. The norm on $X$ is locally octahedral if and only if $g'(X)=2$.
\end{cor}

For weakly octahedral spaces we have the following geometric characterization.

\begin{prop}\label{prop: woh balls}
Let $X$ be a Banach space. The following assertions are equivalent:
\begin{enumerate}[\upshape (i)]
\item the norm on $X$ is weakly octahedral;
\item if $x_1,\dotsc,x_n\in X$ and $r_1,\dotsc,r_n>0$ are such that
\[
S_X\subset \bigcup_{i=1}^n B(x_i,r_i)
\]
then for every $x^\ast\in S_{X^\ast}$ one has
$S_X\subset\{x\in X\colon |x^*(x-x_i)|\leq r_i\}$ for some $i$ in $\{1,\dotsc,n\}$.
\end{enumerate}
\end{prop}
\begin{proof}
(i) $\Rightarrow$ (ii). Assume that the norm on $X$ is weakly octahedral, and $S_X\subset\bigcup_{i=1}^n  B(x_i,r_i)$ for some $x_1,\dotsc,x_n\in X$ and $r_1,\dotsc,r_n>0$. Let $x^*\in S_{X^*}$. We have to show that for some $i\in\{1,\dotsc,n\}$, one has
\[S_X\subset\{x\in X\colon |x^*(x-x_i)|\leq r_i\}.\]
Suppose that, on the contrary, for every $i\in\{1,\dotsc,n\}$ there is an $x\in S_X$ such that $|x^*(x-x_i)|>r_i$. Pick $\varepsilon>0$ satisfying \[r_i<(1-\varepsilon)(1+|x^*(x_i)|)\quad\text{for all $i\in\{1,\dotsc,n\}$.}\] Since the norm on $X$ is weakly octahedral, there is a $y\in S_X$ such that \[(1-\varepsilon)(1+|x^*(x_i)|)\leq\|x_i-y\|\quad\text{for all $i\in\{1,\dotsc,n\}$.}\] This yields  $r_i<\|x_i-y\|$ for every $i\in\{1,\dotsc,n\}$, which is a contradiction.

(ii) $\Rightarrow$ (i). Assume that (ii) holds. Let $E$ be a finite-dimensional subspace of $X$. Let $x^\ast\in S_{X^\ast}$ and let $0<\varepsilon\leq1$. Suppose that, for every $y\in S_X$ there is an $x\in E$ such that
\begin{equation}\label{eq: woh}
\|x-y\|<(1-\varepsilon)(|x^\ast(x)|+1).
\end{equation}
Then $\|x\|<\frac{2-\varepsilon}{\varepsilon}$. Denote by $\delta=\varepsilon/2$. Consider now a finite $\delta$-net $\{x_1,\dots,x_n\}$ for $\frac{2-\varepsilon}{\varepsilon}B_E$. If $y\in S_X$, then find a corresponding $x\in E$ such that (\ref{eq: woh}) holds, and choose $x_i$ such that $\|x-x_i\|<\delta$. By (\ref{eq: woh}), we have
\begin{align*}
\|x_i-y\|&\leq\|x-y\|+\delta\\
&<(1-\varepsilon)(|x^\ast(x)|+1)+\delta\\
&\leq(1-\varepsilon)(|x^\ast(x_i)|+\delta+1)+\delta\\
&=(1-\varepsilon)|x^\ast(x_i)|+1-\varepsilon\delta\\
&\leq(1-\varepsilon^2/2)(|x^\ast(x_i)|+1).
\end{align*}
Thus $S_X\subset \bigcup_{i=1}^n  B(x_i,r_i)$, where $r_i=(1-\varepsilon^2/2)(|x^\ast(x_i)|+1)$.
On the other hand,
\[S_X\not\subset\{x\in X\colon |x^\ast(x-x_i)|\leq r_i\}\quad\text{for all $i\in\{1,\dotsc,n\}$}.\]
This contradicts (ii).
\end{proof}

We recall a simple criterion for the weakly octahedral norm.
\begin{lem}[{\cite[Proposition 2.2]{HLP}}]\label{lem: woh}
Let $X$ be a Banach space. The following assertions are equivalent:
\begin{enumerate}[\upshape (i)]
\item the norm on $X$ is weakly octahedral;
\item for every $n\in\mathbb N$, $x_1,\dotsc,x_n\in S_X$, $x^\ast\in B_{X^\ast}$, and $\eps>0$, there is a $y\in S_X$ such that
\[
\|x_i+ty\|\geq(1-\eps)(|x^\ast(x_i)|+t)\quad\text{for all $i\in\{1,\dotsc,n\}$ and $t>0$}.
\]
\end{enumerate}
\end{lem}
\begin{rem}
It is not clear whether it suffices to take $t=1$ in the preceding condition, i.e. whether condition (ii) is equivalent to the following \begin{enumerate}
\item[(iii)]for every $n\in\mathbb N$, $x_1,\dotsc,x_n\in S_X$, $x^\ast\in B_{X^\ast}$, and $\eps>0$, there is a $y\in S_X$ such that
\[
\|x_i+y\|\geq(1-\eps)(|x^\ast(x_i)|+1)\quad\text{for all $i\in\{1,\dotsc,n\}$}.
\]
\end{enumerate}
However, similarly to the proof of Proposition~\ref{prop: woh balls} we observe that condition (iii) is equivalent to the following
\begin{enumerate}
\item[(iv)] if $x_1,\dotsc,x_n\in S_X$ and $r_1,\dotsc,r_n>0$ are such that
\[
S_X\subset \bigcup_{i=1}^n B(x_i,r_i)
\]
then for every $x^\ast\in S_{X^\ast}$ one has
$S_X\subset\{x\in X\colon |x^*(x-x_i)|\leq r_i\}$ for some $i$ in $\{1,\dotsc,n\}$.
\end{enumerate}
\end{rem}


\begin{thebibliography}{14}
\bibitem{ALN}
T.~{A}brahamsen, V.~{L}ima, and O.~{N}ygaard, \emph{{R}emarks on diameter 2
  properties}, J. Conv. Anal. \textbf{20} (2013), 439--452.

\bibitem{MR1541729}
B.~Bagchi, G.~Misra, N.~S.~N. Sastry, Florida State University Mail
  Room~Problem Group, O.~P. Lossers, and J.~H. Steelman, \emph{Problems and
  {S}olutions: {S}olutions of {A}dvanced {P}roblems: 6577}, Amer. Math. Monthly
  \textbf{97} (1990), 436--437.

\bibitem{MR1775391}
M.~Baronti, E.~Casini, and P.~L. Papini, \emph{On average distances and the
  geometry of {B}anach spaces}, Nonlinear Anal. \textbf{42} (2000), 533--541.

\bibitem{BGLPRZoca}
J.~{B}ecerra {G}uerrero, G.~{L}\'{o}pez {P}\'{e}rez, and A.~{R}ueda {Z}oca,
  \emph{Octahedral norms and convex combination of slices in {B}anach spaces},
  preprint.

\bibitem{MR2984117}
J.~Castillo and P.~L. Papini, \emph{Smallness and the covering of a {B}anach
  space}, Milan J. Math. \textbf{80} (2012), 251--263.

\bibitem{CPS}
J.~{C}astillo, P.~L. {P}apini, and M.~{S}im{\~o}es, \emph{{T}hick coverings for
  the unit ball of a {B}anach space}, preprint.

\bibitem{DGZ}
R.~Deville, G.~Godefroy, and V.~Zizler, \emph{{S}moothness and renormings in
  {B}anach spaces}, Pitman Monographs and Surveys in Pure and Applied
  Mathematics, vol.~64, Longman Scientific \& Technical, Harlow, 1993.

\bibitem{G}
G.~Godefroy, \emph{{M}etric characterization of first {B}aire class linear
  forms and octahedral norms}, Studia Math. \textbf{95} (1989), 1--15.

\bibitem{MR983386}
G.~Godefroy and N.~J. Kalton, \emph{The ball topology and its applications},
  Banach space theory ({I}owa {C}ity, {IA}, 1987), Contemp. Math., vol.~85,
  Amer. Math. Soc., Providence, RI, 1989, pp.~195--237.

\bibitem{HLP}
R.~{H}aller, J.~{L}angemets, and M.~{P}\~oldvere, \emph{{O}n duality of
  diameter $2$ properties}, preprint.

\bibitem{MR2844454}
V.~Kadets, V.~Shepelska, and D.~Werner, \emph{Thickness of the unit sphere,
  {$\ell_1$}-types, and the almost {D}augavet property}, Houston J. Math.
  \textbf{37} (2011), no.~3, 867--878.

\bibitem{MR1735844}
P.~L. Papini, \emph{Average distances and octahedral norms}, Bull. Korean Math.
  Soc. \textbf{36} (1999), 259--272.

\bibitem{MR2583590}
\bysame, \emph{Covering the sphere and the ball in {B}anach spaces}, Commun.
  Appl. Anal. \textbf{13} (2009), 579--586.

\bibitem{MR0228997}
R.~Whitley, \emph{The size of the unit sphere}, Canad. J. Math. \textbf{20}
  (1968), 450--455.

\end{thebibliography}

\end{document}